\documentclass{amsart}

\newtheorem{thm}{Theorem}
\newtheorem{lemma}{Lemma}

\newtheorem{prop}{Proposition}

\newtheorem{lemmax}{Lemma}

\numberwithin{equation}{section}

\usepackage{cite}
\usepackage{hyperref}

\begin{document}

\title{A characterization of the inclusions between mixed norm spaces}

\author{Irina Ar\'evalo}
\address{Departamento de Matem\'aticas, Universidad Aut\'onoma de Madrid, 28049 Madrid, Spain}
%\curraddr{Department of Mathematics and Statistics,
%Case Western Reserve University, Cleveland, Ohio 43403}
\email{irina.arevalo@uam.es}
%\thanks{This work is supported by MICINN Grant BES-2013-062668 and is part of the author's doctoral thesis at Universidad Aut\'onoma de Madrid under the supervision of Professor Dragan Vukoti\'c. She wants to thank him for his useful comments and encouragement.}

%\author{Author Two}
%\address{Mathematical Research Section, School of Mathematical Sciences,
%Australian National University, Canberra ACT 2601, Australia}

%    General info
%, 14E20; Secondary 46E25, 20C20}
\date{\today}



\begin{abstract}
We consider the mixed norm spaces of Hardy type studied by Flett and others. 
%In this work we will improve some growth estimates, both pointwise and in mean, and give a complete characterization of inclusions between these spaces $H(p,q,\alpha).$
%In this work 
We study some properties of these spaces related to mean and pointwise growth and complement some partial results by various authors by giving a complete characterization of the inclusion between $H(p,q,\alpha)$ and $H(u,v,\beta),$ depending on the parameters $p, q, \alpha, u, v$ and $\beta.$ 
\end{abstract}
\maketitle

\section{Introduction}

For $p,q,\alpha>0,$ an analytic function on the unit disk $f$ is said to belong to the mixed norm space $H(p,q,\alpha)$ if and only if $$\|f\|_{p,q,\alpha}^q=\alpha q\int_0^1(1-r)^{\alpha q-1}\left(\int_0^{2\pi}|f(re^{i\theta})|^p\,\frac{d\theta}{2\pi}\right)^{q/p}\,dr<\infty.$$

This expression first appears in Hardy and Littlewood's paper on properties of the integral mean \cite{HL}, but the mixed norm space was not explicitly defined until Flett's works \cite{Flett}, \cite{Flett2}. Since then, these spaces have been studied by many authors (see \cite{aj}, \cite{Blasco}, \cite{BKV}, \cite{Gab}, \cite{Sledd}). Recently, the mixed norm spaces are mentioned in the works \cite{av}, \cite{avetisyan}, which are closely related to the main topic in this paper, see also the forthcoming monograph \cite{JVA}.

The mixed norm spaces form a family of complete spaces that contains the Hardy and Bergman spaces. In the references given, many properties of these spaces have been studied, such as pointwise growth (that will appear later in this work), duality, relation with coefficient multipliers and partial results on inclusions, but, to the best of our knowledge, a complete characterization of the inclusions between different spaces $H(p,q,\alpha)$ has not been recorded. 

In this paper we complete the table of inclusions between different mixed norm spaces by finding a bound for the norm of the inclusion operator whenever an inclusion holds, and by giving explicit examples of functions
to show that no inclusion takes place in some cases. For that, we prove some preliminary results, of interest by themselves, on mean and pointwise growth, norm of the point-evaluation functional and rate of decrease of the integral means. 

%We also give a better estimate of the mean and pointwise growth of functions in these spaces and an approximation to the norm of the point-evaluation functional.

%In this paper we study mean and pointwise growth of functions in these spaces and give an approximation to the norm of the point-evaluation functional before completing the table of inclusions between different mixed-norm spaces. To the best of our knowledge, this has not been done before, although there are scattered results in the literature. 
 
%In this paper, after section 2, which is dedicated to preliminaries, we give a better estimate of the mean and pointwise growth of functions in these spaces and an approximation to the norm of the point-evaluation functional. In the last section we formulate and prove the main theorems, which give a complete description of the inclusions between $H(p,q,\alpha)$ and $H(u,v,\beta),$ depending on the parameters $p, q, \alpha, u, v, \beta.$ We prove several results on bounds for the integral means of functions in the mixed norm spaces, and give explicit examples of functions when there is no inclusion. 

From now on, we will understand $1/\infty$ as zero, the letters $A, B, C, C', K, m$ will be positive constants, and we will say that two quantities are comparable, denoted by $\alpha\approx\beta,$ if there exist two positive constants $C$ and $C'$ such that $$C\alpha\leq\beta\leq C'\alpha.$$

The present work was partially supported by MINECO grant MTM2012-37436-C02-02, Spain and is part of the author's doctoral thesis work at Universidad Aut\'onoma de Madrid under the supervision of Professor Dragan Vukoti\'c. She wants to thank him for his useful comments and encouragement, and Marcos de la Oliva for finding a big mistake in a previous version. 

\section{Preliminaries}

Let $\mathcal{H}(\mathbb{D})$ be the space of analytic functions on the disk $\mathbb{D}=\{z\in\mathbb{C}:|z|<1\}.$ For $f\in \mathcal{H}(\mathbb{D})$ and $r\in(0,1)$ let $M_p(r,f)$ be the integral mean $$M_p(r,f)=\left(\int_0^{2\pi}|f(re^{i\theta})|^p\,\frac{d\theta}{2\pi}\right)^{1/p}$$ if $0<p<\infty$ and $$M_\infty(r,f)=\max_{0\leq\theta<2\pi}|f(re^{i\theta})|.$$ 

We consider the spaces $H(p, q, \alpha),$ $0<p,q\leq\infty,$ $0<\alpha<\infty,$ consisting of analytic functions on $\mathbb{D}$ such that $$\|f\|_{p,q,\alpha}^q=\alpha q\int_0^1(1-r)^{\alpha q-1}M_p^q(r,f)\,dr<\infty,$$ if $q<\infty,$ and $$\|f\|_{p,\infty,\alpha}=\sup_{0\leq r<1}(1-r)^\alpha M_p(r,f)<\infty.$$ 

For any $0<p,q\leq\infty,$ $0<\alpha<\infty$ the space $H(p,q,\alpha)$ is a complete subspace of the space $L(p,q,\alpha)$ of measurable functions in $\mathbb{D}$ (see \cite{ben}).

In particular, one can identify the weighted Bergman space $A^p_\alpha,$ $0<p<\infty,$ $-1<\alpha<\infty,$ of analytic functions on the unit disk such that $$\int_{\mathbb{D}}|f(z)|^p(1-|z|^2)^\alpha\,dA(z)<\infty$$ with the space $H\left(p,p,\frac{\alpha+1}{p}\right)$ and the Hardy space $H^p$ of functions in $\mathcal{H}(\mathbb{D})$ for which $$\sup_{0<r<1}M_p(r,f)<\infty$$ with $H(p,\infty,0).$ The mixed norm spaces are also related to other spaces of analytic functions, such as Besov and Lipschitz spaces, via fractional derivatives (see \cite[Chapter 7]{JVA}).

Familiar examples of analytic functions on the unit disk are the functions of type $(1-z)^{-\gamma},$ with $\gamma$ a real constant. It is well known that such function is in the Hardy space $H^p$ if and only if $\gamma<1/p$ and in the Bergman space $A^p$ if and only if $\gamma<2/p.$ The following lemma determines when these functions belong to $H(p,q,\alpha)$ (see \cite{av}).

\begin{lemmax}\label{fg}
Let $0<p\leq\infty,$ $0<\alpha<\infty.$ The functions $f(z)=\frac{1}{(1-z)^\gamma}$ belong to $H(p,q,\alpha),$ $0<q<\infty,$ if and only if $\gamma<\alpha+1/p,$ and to $H(p,\infty,\alpha)$ if and only if $\gamma\leq\alpha+1/p.$
\end{lemmax}

Starting with these examples we can search for functions with faster growth for $z\in\mathbb{R},$ $0<z<1.$ The following lemma gives us examples of functions which attain the critical exponent shown in the last lemma, but still belong to the space (see \cite{av}).

\begin{lemmax}\label{fc}
Let $0<p\leq\infty,$ $0<\alpha<\infty.$ The functions $$f(z)=\frac{1}{(1-z)^{\alpha+1/p}}\left(\log\frac{e}{1-z}\right)^{-c}$$ belong to $H(p,q,\alpha)$ if and only if $c>1/q$ for $q<\infty,$ and $c\geq0$ for $q=\infty.$
\end{lemmax}

Another well-known class of analytic functions is the class of lacunary series. Such series belongs to the Hardy space $H^p$ if and only if the sequence formed with its coefficients belongs to the $l^2$ space. In that case (and only then) the function has radial limits almost everywhere, and otherwise, has radial limits almost nowhere. The following result appears in \cite[Thm. 8.1.1]{JVA}, based on \cite{MP}.  

\begin{lemmax}\label{lag}
Let $f(z)=\sum_{n=1}^\infty a_n\,z^{2^{n-1}}$ and $0<p,q\leq\infty,$ $0<\alpha<\infty.$ Then $f\in H(p,q,\alpha)$ if and only if $\{2^{-n\alpha}a_n\}\in l^q.$ 
\end{lemmax}

In particular, there are functions with radial limits almost nowhere in each $H(p,q,\alpha)$ with $\alpha>0$ (for instance, the lacunary series with coefficients equal to $1$ satisfies $\sum_{n=0}^\infty2^{-n\alpha q}|a_n|^q<\infty$ for every $0<p,q\leq\infty,$ $0<\alpha<\infty,$ but $\sum_{n=0}^\infty|a_n|^2=\infty$). Therefore, the Hardy space does not contain any $H(p,q,\alpha)$ with $\alpha>0.$

\section{Pointwise and mean estimates}

If $f$ is a function in $H(p,q,\alpha),$ we have the following estimate for its integral means. 

\begin{lemma}\label{mediaintegral}
If $f\in H(p, q, \alpha),$ $0<p\leq\infty,$ $0<q, \alpha<\infty,$ then
$$M_p(r,f)=o\left((1-r)^{-\alpha}\right)$$ as $r\to 1.$
\end{lemma}

\begin{proof}

Since the integral $$\alpha q\int_0^r(1-\rho)^{\alpha q-1}M_p^q(\rho,f)\,d\rho$$
converges to $\|f\|_{p,q,\alpha}^q$ as $r\to 1,$ then for every $\varepsilon>0$ there exists $r_0$ such that  
\begin{equation}\label{varepsilon}\alpha q\int_r^1(1-\rho)^{\alpha q-1}M_p^q(\rho,f)\,d\rho<\varepsilon\end{equation} for
every $r>r_0.$
Therefore, since the integral means are increasing as functions of $r,$ we get
\begin{align*}
(1-r)^{\alpha q}M_p^q(r,f)&=\alpha q\int_{r}^1(1-\rho)^{\alpha q-1}M_p^q(r,f)\,d\rho\\&\leq\alpha q\int_r^1(1-\rho)^{\alpha q-1}M_p^q(\rho,f)\,d\rho<\varepsilon.
\end{align*}
\end{proof}

Moreover, it follows from the proof that if $f\in H(p,q,\alpha),$ then \begin{equation}\label{alpha}M_p(r,f)\leq\frac{\|f\|_{p,q,\alpha}}{(1-r)^{\alpha}}\end{equation} since we can bound the integral in (\ref{varepsilon}) by the norm of $f$ instead of $\varepsilon.$ Notice that, taking supremum over $r,$ we get \begin{equation}\label{inftynorm}\|f\|_{p,\infty,\alpha}\leq\|f\|_{p,q,\alpha},\end{equation} and therefore $H(p,q,\alpha)\subseteq H(p,\infty,\alpha)$ for every $0<p,q\leq\infty,$ $0<\alpha<\infty.$
%since
%\begin{align*}
%\|f\|_{p,q,\alpha}^q&=\alpha q\int_0^1(1-s)^{\alpha q-1}M_p^q(s,f)\,ds\geq\alpha q\int_r^1(1-s)^{\alpha q-1}M_p^q(s,f)\,ds\\&\geq\alpha q\int_{r}^1(1-s)^{\alpha q-1}M_p^q(r,f)\,ds=(1-r)^{\alpha q}M_p^q(r,f).
%\end{align*}

Although the result in Lemma~\ref{mediaintegral} fails for $q=\infty$ as the function $f(z)=(1-~z)^{-\alpha-1/p}$ shows, the above bound for the integral mean still holds since $$\|f\|_{p,\infty,\alpha}=\sup_{0\leq \rho<1}(1-\rho)^\alpha M_p(\rho,f)\geq(1-r)^\alpha M_p(r,f)$$ for any $r,$ $0<r<1,$  and therefore \begin{equation}\label{infty}M_p(r,f)\leq\frac{\|f\|_{p,\infty,\alpha}}{(1-r)^{\alpha}}\end{equation} for $f\in H(p,\infty,\alpha).$

For the Bergman spaces $A^p,$ besides the well-known big-Oh growth inequality, we have the estimate $|f(z)|=o\left((1-|z|)^{2/p}\right)$ as $|z|\to 1$ for every $f\in A^p.$ This is a consequence of the subharmonicity of $|f|^p$ and the inequality: $$\int_{D(a,r)}|f(z)|^p\,dA(z)\leq\int_\mathbb{D}|f(z)|^p\,dA(z)=\|f\|_{A^p}^p$$ for $a\in\mathbb{D}$ and $r<1$ (see \cite[Page 7]{DurSchu}). We can obtain an analogous result for the Hardy spaces with similar techniques that cannot be used in the mixed norm spaces. However, the result still holds, as we shall show next.

\begin{prop} \label{acotacion}
If $f\in H(p, q, \alpha),$ $0<p\leq\infty,$ $0<q, \alpha<\infty,$ then 
$$|f(z)|=o\left((1-|z|)^{\alpha+1/p}\right)$$ as $|z|\to 1.$
\end{prop}

In the proof we will use the following identity.

\begin{lemma}\label{lemma}
For $0<p,q,\alpha<\infty$ and $z\in\mathbb{D},$  
\begin{align*}
\int_{|z|}^1(1-\rho)^{\alpha q-1}(\rho-|z|)^{q/p}\,d\rho=B(\alpha q,q/p+1)\,(1-|z|)^{\alpha q+q/p},
\end{align*}
where $B(a,b)=\int_0^1(1-x)^{a-1}x^{b-1}\,dx,$ $a,b>0,$ is the Beta function.
\end{lemma}

\begin{proof}
With the change of variables $x=\frac{\rho-|z|}{1-|z|},$ 
\begin{align*}\int_{|z|}^1(1-\rho)^{\alpha q-1}(\rho-|z|)^{q/p}\,d\rho&=\int_0^1(1-x)^{\alpha q-1}(1-|z|)^{\alpha q-1}x^{q/p}(1-|z|)^{q/p}(1-|z|)\,dx\\&=(1-|z|)^{\alpha q+q/p}\int_0^1(1-x)^{\alpha q-1}x^{q/p}\,dx.\end{align*}
\end{proof}

Next, we prove Proposition~\ref{acotacion}.

\begin{proof}[Proof of Proposition~\ref{acotacion}]
If $p=\infty,$ it is easy to see that, for $r$ close enough to 1 (as in Lemma~\ref{mediaintegral}),
\begin{align}
|f(re^{i\theta})|^q(1-r)^{\alpha q}&=\alpha q\,|f(re^{i\theta})|^q\int_r^1(1-\rho)^{\alpha q-1}\,d\rho\\&\leq\alpha q\int_r^1(1-\rho)^{\alpha q-1}M_\infty^q(\rho,f)\,d\rho<\varepsilon.
\notag\end{align}
%Let $f$ be a function in $H(p,q,\alpha).$ As in the proof of Lemma~\ref{mediaintegral}, for every $\varepsilon>0$ there exists $r_0,$ $0<r_0<1,$ such that for every $r>r_0,$
%\begin{equation}\label{i1}\varepsilon>\alpha q\int_r^1(1-s)^{\alpha q-1}M_p^q(s,f)\,ds,\end{equation}
%Now we estimate the mean integral $M_p(r,f)$ using the Poisson integral:
If $0<p<\infty,$ we estimate the integral mean $M_p(r,f)$ using the Poisson integral: 

%If $r<s<1,$ then $\frac{r}{s}<1.$ Let $f\in H(p,q,\alpha)$ and $f_s(z)=f(sz),$ the  $f_s\in H^\infty$ and we can use the well-known inequality for functions in $H^p$ (\cite[Page 28]{Dur}) 
%$$|f(\rho e^{i\theta})|^p\leq\frac{1}{2\pi}\int_0^{2\pi}|f(e^{it})|^p\frac{1-\rho^2}{|1-\rho e^{i(\theta-t)}|^2}\,dt.$$ 
%with $f_s$ and $\rho=\frac{r}{s}:$
%\begin{align*}\left|f_s\left(\frac{r}{s}e^{i\theta}\right)\right|^p&=|f(re^{i\theta})|^p\leq\frac{1}{2\pi}\int_0^{2\pi}|f_s(e^{it})|^p\frac{1-\left(\frac{r}{s}\right)^2}{|1-\frac{r}{s}e^{i(\theta-t)}|^2}\,dt\\&=\frac{1}{2\pi}\int_0^{2\pi}|f(se^{it})|^p\frac{s^2-r^2}{|s-re^{i(\theta-t)}|^2}\,dt\end{align*}

%It follows, 
Let $\rho\in(0,1)$ and define $f_\rho(z)=f(\rho z),$ for $f\in\mathcal{H}(\mathbb{D})$ and $z\in\mathbb{D}.$ Since $f_\rho\in H^\infty$ for any $f\in H(p,q,\alpha)$ and $r<\rho$ we have, as in \cite{Gab},
\begin{align*}|f(re^{i\theta})|^p&\leq\frac{1}{2\pi}\int_0^{2\pi}|f(\rho e^{it})|^p\frac{\rho^2-r^2}{|\rho-re^{i(\theta-t)}|^2}\,dt\leq\frac{1}{2\pi}\int_0^{2\pi}|f(\rho e^{it})|^p\frac{\rho^2-r^2}{(\rho-r)^2}\,dt\\&\leq \frac{2}{\rho-r}\frac{1}{2\pi}\int_0^{2\pi}|f(\rho e^{it})|^p\,dt=\frac{2}{\rho-r}\,M_p^p(\rho,f).\end{align*}
Hence, \begin{equation}\label{acot} |f(re^{i\theta})|(\rho-r)^{1/p}\leq 2^{1/p}M_p(\rho,f).\end{equation}
Let $\varepsilon>0,$ then for $r$ close to 1 we have, using the above estimate and Lemma~\ref{lemma}, as in (\ref{varepsilon}),
%Using 
%this estimation in (\ref{i1}) and 
%the identity in Lemma~\ref{lemma}, as in (\ref{varepsilon}), for $\varepsilon$ small enough,
\begin{gather}\label{varepsilon2}
\frac{\alpha q}{2^{q/p}}\,B(\alpha q,q/p+1)\,|f(re^{i\theta})|^q(1-r)^{\alpha q+q/p}\\=\frac{\alpha q}{2^{q/p}}|f(re^{i\theta})|^q\int_r^1(1-\rho)^{\alpha q-1}(\rho-r)^{q/p}\,d\rho\notag\\\leq\alpha q\int_r^1(1-\rho)^{\alpha q-1}M_p^q(\rho,f)\,d\rho<\varepsilon.
\notag\end{gather}
\end{proof}

From the above proof, the following known pointwise estimate also follows: if we bound the integral in (\ref{varepsilon2}) by $\|f\|_{p,q,\alpha},$
%\begin{align*}
%\|f\|_{p,q,\alpha}^q&=\alpha q\int_0^1(1-s)^{\alpha q-1}M_p^q(s,f)\,ds\geq\alpha q\int_r^1(1-s)^{\alpha q-1}M_p^q(s,f)\,ds,\\&\geq \frac{\alpha q}{2^{q/p}}|f(re^{i\theta})|^q\int_r^1(1-s)^{\alpha q-1}(s-r)^{q/p}\,ds\\&=\frac{\alpha q}{2^{q/p}}\,B(\alpha q,q/p+1)\,|f(re^{i\theta})|^q(1-r)^{\alpha q+q/p},\end{align*}
then \begin{equation}\label{puntual}|f(z)|\leq m\frac{\|f\|_{p,q,\alpha}}{(1-|z|)^{\alpha+1/p}},\end{equation} with 

\begin{equation}\label{m}m=
\begin{cases}
{\displaystyle\frac{2^{1/p}}{(\alpha q\,B(\alpha q,q/p+1))^{1/q}}}\text{\,\quad if } p<\infty\\
\\
\quad\quad\quad\ \quad1\text{\quad\quad\quad\quad\quad\quad\ if } p=\infty.
\end{cases}\end{equation}

One should notice that, once again, this Proposition does not hold for $q=\infty,$ as the function $f(z)=(1-z)^{-\alpha-1/p}$ shows. However, the pointwise estimation is still true (see \cite[Prop. 7.1.1]{JVA}).

If we denote by $\phi_z$ the point-evaluation functional: $$\phi_z(f)=f(z),$$ for $f\in H(p,q,\alpha),$ and $z\in\mathbb{D},$ it is easy to see that its norm can be estimated as follows: \begin{equation}\label{functional}\|\phi_z\|\leq \frac{m}{(1-|z|)^{\alpha+\frac{1}{p}}},\end{equation} with $m$ as in (\ref{m}).

Now, for a given $z$ in $\mathbb{D},$ we will find a function $f_z$ in $H(p,q,\alpha)$ with pointwise growth of maximal order, that is: $\|f_z\|\approx 1$ and $$|f_z(z)|\approx\|\phi_z\|\approx\frac{1}{(1-|z|)^{\alpha+1/p}}.$$
Here we give a general version of a well-known fact for Bergman spaces.

\begin{prop}\label{igualdad}
For $z\in\mathbb{D},$ $0<p,q\leq\infty,$ $0<\alpha<\infty$ and $s>0,$ the functions $$f_z(w)=\frac{(1-|z|^2)^s}{(1-\bar{z}w)^{\alpha+\frac{1}{p}+s}}$$ satisfy $|f_z(z)|\approx\|\phi_z\|$ and $\|f_z\|_{p,q,\alpha}\approx 1.$
\end{prop}

\begin{proof}

First we check that $f_z$ belongs to $H(p,q,\alpha)$ and estimate its norm:

If $w=re^{i\theta},$ then (see \cite[Page 65]{Dur})
\begin{align*}M_p^p(r,f_z)&=\int_0^{2\pi}\frac{(1-|z|^2)^{ps}}{|1-\bar{z}re^{i\theta}|^{p(\alpha+\frac{1}{p}+s)}}\frac{d\theta}{2\pi}\approx\frac{(1-|z|^2)^{ps}}{(1-r|z|)^{p(\alpha+\frac{1}{p}+s)-1}}=\frac{(1-|z|^2)^{ps}}{(1-r|z|)^{(\alpha+s)p}},
\end{align*}
for $p<\infty$ and $$M_\infty(r,f_z)\approx\frac{(1-|z|^2)^{s}}{(1-r|z|)^{\alpha+s}}.$$
Therefore, if $q<\infty$ and $0<p\leq\infty,$
\begin{align*}\|f_z\|^q_{p,q,\alpha}&=\alpha q\int_0^1(1-r)^{\alpha q-1}M_p^q(r,f_z)\,dr\\&\approx\alpha q\,(1-|z|)^{sq}\int_0^1(1-r)^{\alpha q-1}\frac{1}{(1-r|z|)^{(\alpha+s)q}}\,dr.\\
\end{align*}
Now, on the one hand,
\begin{align*}\|f_z\|^q_{p,q,\alpha}&\approx\alpha q\,(1-|z|)^{sq}\int_0^1(1-r)^{\alpha q-1}\frac{1}{(1-r|z|)^{(\alpha+s)q}}\,dr\\&\geq\alpha q\,(1-|z|)^{sq}\int_{|z|}^1(1-r)^{\alpha q-1}\frac{1}{(1-r|z|)^{(\alpha+s)q}}\,dr\\&\geq\alpha q\,\frac{(1-|z|)^{sq}}{(1-|z|^2)^{(\alpha+s)q}}\int_{|z|}^1(1-r)^{\alpha q-1}\,dr\\&\approx\frac{1}{(1-|z|)^{\alpha q}}(1-|z|)^{\alpha q}=1
\end{align*}
and, on the other hand, integrating by parts and using $(1-r)^{\alpha q}\leq(1-r|z|)^{\alpha q},$ 
\begin{align*}
\|f_z\|^q_{p,q,\alpha}&\approx(1-|z|)^{sq}\int_0^1\alpha q\,(1-r)^{\alpha q-1}\frac{1}{(1-r|z|)^{(\alpha+s)q}}\,dr\\&=(1-|z|)^{sq}\left(1-(\alpha+s)q\,|z|\,\int_0^1(1-r)^{\alpha q}\frac{1}{(1-r|z|)^{^{(\alpha+s)q+1}}}\,dr\right)\\&\leq(1-|z|)^{sq}\left(1-(\alpha+s)q\,|z|\,\int_0^1(1-r|z|)^{-(sq+1)}\,dr\right)\\&=(1-|z|)^{sq}\left(1-\frac{\alpha+s}{s}\,\left((1-|z|)^{-sq}-1\right)\right)\\&=\left(1+\frac{\alpha+s}{s}\right)(1-|z|)^{sq}-\frac{\alpha+s}{s}\approx 1.
%1+\left(1-\frac{\alpha+s}{s}\right)(1-|z|)^{sq}\approx 1.
\end{align*}

If $q=\infty,$ 
\begin{align*}
\|f_z\|_{p,\infty,\alpha}=\sup_{0\leq r<1}(1-r)^\alpha M_p(r,f_z)\approx\sup_{0\leq r<1}(1-r)^\alpha\frac{(1-|z|)^s}{(1-r|z|)^{\alpha+s}}.
\end{align*}
Since $1-r\leq1-r|z|$ and $1-|z|\leq1-r|z|,$ we have $$\|f_z\|_{p,\infty,\alpha}\approx\sup_{0\leq r<1}(1-r)^\alpha\frac{(1-|z|)^s}{(1-r|z|)^{\alpha+s}}\leq\sup_{0\leq r<1}(1-r|z|)^\alpha\frac{(1-r|z|)^s}{(1-r|z|)^{\alpha+s}}=1$$
hence
\begin{align*}\|f_z\|_{p,\infty,\alpha}&\geq\sup_{|z|<r<1}(1-r)^\alpha\frac{(1-|z|)^s}{(1-r|z|)^{\alpha+s}}\geq\frac{(1-|z|)^s}{(1-|z|^2)^{\alpha+s}}\sup_{|z|<r<1}(1-r)^\alpha\\&=\frac{(1-|z|)^s(1-|z|)^\alpha}{(1-|z|^2)^{\alpha+s}}\approx1.\end{align*}

%\sup_{r<|z|}(1-|z|)^\alpha\frac{(1-|z|)^s}{(1-|z|^2)^{\alpha+s}}\approx 1.$$
Now that we know that $f_z\in H(p,q,\alpha),$ we see easily that
$$|f_z(z)|=\frac{(1-|z|^2)^s}{(1-|z|^2)^{\alpha+\frac{1}{p}+s}}=\frac{1}{(1-|z|^2)^{\alpha+\frac{1}{p}}}\approx\frac{1}{(1-|z|)^{\alpha+\frac{1}{p}}}$$
and from here
$$\|\phi_z\|\approx\|\phi_z\|\|f_z\|_{p,q,\alpha}\geq|f_z(z)|\approx\frac{1}{(1-|z|)^{\alpha+\frac{1}{p}}}.$$
With (\ref{functional}), we get $$|f_z(z)|\approx\|\phi_z\|\approx\frac{1}{(1-|z|)^{\alpha+\frac{1}{p}}}.$$
\end{proof}

\section{Inclusions between mixed norm spaces}

The main theorems in this work, which characterize completely the inclusions between mixed norm spaces, are the following. To avoid repetitions, we recall here that we are assuming our parameters to be $0<\alpha, \beta<\infty$ and $0<p,q,u,v\leq\infty.$

\begin{thm}\label{inclusiones1} 
If $p\geq u,$ then
$$H(p,q,\alpha)\subseteq H(u,v,\beta)\Leftrightarrow
\begin{cases}
\alpha<\beta\text{\quad or}\\
\alpha=\beta\text{\quad and\quad}q\leq v.
\end{cases}
$$
\end{thm}

\begin{thm}\label{inclusiones2}
If $p<u,$ then
$$H(p,q,\alpha)\subseteq H(u,v,\beta)\Leftrightarrow
\begin{cases}
\alpha+\frac{1}{p}<\beta+\frac{1}{u}\text{\quad or}\\
\alpha+\frac{1}{p}=\beta+\frac{1}{u}\text{\quad and\quad}q\leq v.
\end{cases}
$$
\end{thm}

It is worth noticing that we need $\alpha$ to be greater than zero as we stated when these spaces were defined. In the limit case $\alpha=0,$ by a theorem by Hardy and Littlewood (related to the Isoperimetric Inequality, see \cite{matiso}, \cite{isop}), we have $H^p\subseteq A^{2p}.$ That is, $H(p,\infty,0)\subseteq H(2p,2p,1/2p),$ although these parameters do not satisfy Theorem 2. 

Notice also that it is only to be expected that the relation between the spaces would depend on the relation between the parameters $p$ and $u,$ since, ultimately, in order to compare the different spaces we need to compare the sizes of the integral means. In turn, the integral means relate in a different fashion according to the parameters $p$ and $u.$

Therefore, in order to prove these theorems we will need the following estimates of the integral means, which can be found in the literature (see \cite[Thm. 5.9]{Dur}, \cite{HL}).

\begin{lemmax}\label{acotacionmedia}
If $f\in H(p,q,\alpha)$ and $q\leq v<\infty,$ then $$M_p^v(r,f)\leq \|f\|^{v-q}_{p,q,\alpha}\,(1-r)^{-\alpha(v-q)}M_p^q(r,f).$$ 
\end{lemmax}

\begin{proof} 
If $f\in H(p,q,\alpha),$ by the bound on the integral mean (\ref{alpha})  $$M_p(r,f)\leq \|f\|_{p,q,\alpha}\,(1-r)^{-\alpha},$$ and since $q\leq v<\infty,$ 
$$M_p^v(r,f)=M_p^{v-q}(r,f)M_p^q(r,f)\leq \|f\|^{v-q}_{p,q,\alpha}\,(1-r)^{-\alpha(v-q)}M_p^q(r,f).$$
\end{proof}

If $f$ belongs to $H(p,q,\alpha)$ and $u>p$ we have the following bound for $M_u(r,f).$

\begin{lemmax}\label{medias}
If $f\in H(p,q,\alpha)$ and $p<u,$ then $$M_u(r,f)\leq m^{1-\frac{p}{u}}\,\|f\|_{p,q,\alpha}\,(1-r)^{-\alpha+\frac{1}{u}-\frac{1}{p}},$$ where $$m=\frac{2^{1/p}}{(\alpha q\,B(\alpha q,q/p+1))^{1/q}}.$$
\end{lemmax}

\begin{proof}
The pointwise inequality (\ref{puntual}) $$M_\infty(r,f)\leq m\,\|f\|_{p,q,\alpha}\,(1-r)^{-\alpha-\frac{1}{p}}$$ is the case $u=\infty.$
Now if $u<\infty,$ \begin{align}\label{mu}M_u(r,f)&=\left(\int_0^{2\pi}|f(re^{i\theta})|^{u-p}|f(re^{i\theta})|^{p}\,\frac{d\theta}{2\pi}\right)^{1/u}\leq M_\infty^{1-\frac{p}{u}}(r,f)\,M_p^{\frac{p}{u}}(r,f)\\&\leq  m^{1-\frac{p}{u}}\,\|f\|^{1-\frac{p}{u}}_{p,q,\alpha}\,(1-r)^{(1-\frac{p}{u})(-\alpha-\frac{1}{p})}\|f\|^{\frac{p}{u}}_{p,q,\alpha}\,(1-r)^{-\alpha\frac{p}{u}}\notag\\&= m^{1-\frac{p}{u}}\,\|f\|_{p,q,\alpha}\,(1-r)^{-\alpha+\frac{1}{u}-\frac{1}{p}}.\notag
\end{align}
\end{proof}

The growth property for the integral means is well known: for an analytic function $f$ on the disk, $M_p(r,f)\leq M_u(r,f)$ when $p\leq u.$ Furthermore, if $f$ belongs to $H(p,q,\alpha),$ we also have the following property that quantifies how the integral means decrease. 

\begin{lemmax}[Lemma 5, \cite{ArSha}]\label{medias2}
If $0<p\leq u\leq\infty,$ then $$M_u(r,f)\leq C(1-r)^{\frac{1}{u}-\frac{1}{p}} M_p(r,f).$$
\end{lemmax}

Now we can prove the theorems.

\begin{proof}[Proof of Theorem~\ref{inclusiones1}]
Throughout this proof, we will assume that $p\geq u.$ The key to proving the sufficiency is the inequality of the integral means: if $p\geq u,$ then $M_u(r,f)\leq M_p(r,f)$. 
%First, we will see that, if $p\geq u,$ $\alpha<\beta$ or $\alpha=\beta$ and $q\leq v,$ and $f\in H(p,q,\alpha),$ then $f\in H(u,v,\beta).$ 

We suppose first that $\alpha<\beta.$ Then, since $M_p(r,f)\leq \|f\|_{p,q,\alpha}(1-r)^{-\alpha}$ by (\ref{alpha}), we have that, if $v$ is finite,
\begin{align*}\|f\|_{u,v,\beta}^v&=\beta v\int_0^1(1-r)^{\beta v-1}M_u^v(r,f)\,dr\leq \beta v\int_0^1(1-r)^{\beta v-1}M_p^v(r,f)\,dr\\&\leq \beta v\,\|f\|^v_{p,q,\alpha}\int_0^1(1-r)^{\beta v-1}(1-r)^{-\alpha v}\,dr\\&=\beta v\,\|f\|^v_{p,q,\alpha}\int_0^1(1-r)^{v(\beta-\alpha)-1}\,dr=\frac{\beta}{\beta-\alpha}\|f\|^v_{p,q,\alpha}, 
\end{align*} 
and, by (\ref{inftynorm}), 
\begin{align*}\|f\|_{u,\infty,\beta}&=\sup_{0\leq r<1}(1-r)^\beta M_u(r,f)\leq\sup_{0\leq r<1}(1-r)^\alpha M_p(r,f)=\|f\|_{p,\infty,\alpha}\leq\|f\|_{p,q,\alpha}.\end{align*}
Therefore, $f\in H(u,v,\beta)$ for every $f\in H(p,q,\alpha).$

Now, if $\alpha=\beta$ and $q\leq v,$ by Lemma~\ref{acotacionmedia}, 
\begin{align*}\|f\|_{u,v,\beta}^v&=\beta v\int_0^1(1-r)^{\beta v-1}M_u^v(r,f)\,dr\leq \beta v\int_0^1(1-r)^{\beta v-1}M_p^v(r,f)\,dr\\&\leq \beta v\,\|f\|^{v-q}_{p,q,\alpha}\,\int_0^1(1-r)^{\beta v-1}(1-r)^{-\alpha(v-q)}M_p^q(r,f)\,dr\\&=\beta v\,\|f\|^{v-q}_{p,q,\alpha}\,\int_0^1(1-r)^{\alpha q-1}M_p^q(r,f)\,dr\\&=\frac{\beta v}{\alpha q}\,\|f\|^{v-q}_{p,q,\alpha}\,\|f\|_{p,q,\alpha}^q=\frac{v}{q}\,\|f\|^v_{p,q,\alpha}
\end{align*}
if $v<\infty,$ and, again by (\ref{inftynorm}), 
\begin{align*}\|f\|_{u,\infty,\beta}&=\sup_{0\leq r<1}(1-r)^\beta M_u(r,f)\leq\sup_{0\leq r<1}(1-r)^\alpha M_p(r,f)=\|f\|_{p,\infty,\alpha}\leq\|f\|_{p,q,\alpha}.\end{align*}

Hence, in both cases $H(p,q,\alpha)\subseteq H(u,v,\beta),$ and the sufficiency is proven. 

For the necessity, we need to see that $H(p,q,\alpha)\not\subseteq H(u,v,\beta)$ when the parameters do not relate as in the statement of the theorem. 
%if $\alpha>\beta$ or $\alpha=\beta$ and $q>v.$ 
For this, consider a function of type $f(z)=\sum_{n=1}^\infty a_n\,z^{2^{n-1}}$ as in Lemma~\ref{lag}. Recall that $f$ belongs to $H(p,q,\alpha)$ if and only if $\left\{2^{-\alpha n}a_n\right\}\in l^q.$

If $\alpha>\beta,$ let $f(z)=\sum_{n=1}^\infty 2^{n\beta}\,z^{2^{n-1}}.$ Since $$\left\{\frac{a_n}{2^{\alpha n}}\right\}=\left\{\frac{1}{2^{n(\alpha-\beta)}}\right\}\in l^q,$$ the function $f$ belongs to $H(p,q,\alpha),$ but $$\left\{\frac{a_n}{2^{\beta n}}\right\}=\left\{1\right\}\not\in l^v,$$ so this function does not belong to $H(u,v,\beta),$ and therefore $H(p,q,\alpha)\not\subseteq H(u,v,\beta)$ if $\alpha>\beta.$

If $\alpha=\beta$ and $q>v,$ we take $f(z)=\sum_{n=1}^\infty 2^{n\alpha}\,n^{-1/v}\,z^{2^{n-1}}.$ Similarly, $$\left\{\frac{a_n}{2^{\alpha n}}\right\}=\left\{\frac{2^{n\alpha}\,n^{-1/v}}{2^{n\alpha}}\right\}=\left\{\frac{1}{n^{1/v}}\right\}\in l^q,$$ and $f\in H(p,q,\alpha),$ but $$\left\{\frac{a_n}{2^{\beta n}}\right\}=\left\{\frac{2^{n\alpha}\,n^{-1/v}}{2^{\beta n}}\right\}=\left\{\frac{1}{n^{1/v}}\right\}\not\in l^v,$$ so it does not belong to $H(u,v,\beta),$ and hence $H(p,q,\alpha)\not\subseteq H(u,v,\beta)$ if $\alpha=\beta$ and $q>v.$

\end{proof}

\begin{proof}[Proof of Theorem~\ref{inclusiones2}]

As in the last proof, from now on we will assume $p<u.$ Firstly we shall see that if $f\in H(p,q,\alpha)$ and the parameters are ordered as in the statement, then $f\in H(u,v,\beta).$ 
 %$p<u,$ $\alpha<\beta+\frac{1}{u}-\frac{1}{p}$ or $\alpha=\beta+\frac{1}{u}-\frac{1}{p}$ and $q\leq v$ and $f\in H(p,q,\alpha),$ then $f\in H(u,v,\beta).$ 
%Since $p<u$, we will use Lemma~\ref{medias} and Proposition~\ref{LB}.

If $\alpha<\beta+\frac{1}{u}-\frac{1}{p},$ by Lemma~\ref{medias},
\begin{align*}\|f\|_{u,v,\beta}^v&=\beta v\int_0^1(1-r)^{\beta v-1}M_u^v(r,f)\,dr\\&\leq\beta v\, m^{v(1-\frac{p}{u})}\,\|f\|^{v}_{p,q,\alpha}\,\int_0^1(1-r)^{\beta v-1}(1-r)^{v(-\alpha+\frac{1}{u}-\frac{1}{p})}\,dr\\&=\beta v\, m^{v(1-\frac{p}{u})}\,\|f\|^{v}_{p,q,\alpha}\,\int_0^1(1-r)^{v(\beta-\alpha+\frac{1}{u}-\frac{1}{p}) -1}\,dr=\frac{\beta\, m^{v(1-\frac{p}{u})}}{\beta-\alpha+\frac{1}{u}-\frac{1}{p}}\|f\|^{v}_{p,q,\alpha}
\end{align*}
for $v<\infty,$ and \begin{align*}\|f\|_{u,\infty,\beta}&=\sup_{0\leq r<1}(1-r)^\beta M_u(r,f)\\&\leq  m^{1-\frac{p}{u}}\,\|f\|_{p,q,\alpha}\sup_{0\leq r<1}(1-r)^\beta(1-r)^{-\alpha+\frac{1}{u}-\frac{1}{p}}=m^{1-\frac{p}{u}}\,\|f\|_{p,q,\alpha}.\end{align*}

If $\alpha=\beta+\frac{1}{u}-\frac{1}{p}$ and $q\leq v,$ by Lemmas \ref{medias2} and \ref{acotacionmedia},
\begin{align*}\|f\|_{u,v,\beta}^v&=\beta v\int_0^1(1-r)^{\beta v-1}M_u^v(r,f)\,dr\\&\leq\beta v\, C^v\int_{0}^1(1-r)^{\beta v-1}(1-r)^{v(\frac{1}{u}-\frac{1}{p})} M_p^v(r,f)\,dr\\&=\beta v\,C^v\int_{r}^1(1-r)^{\alpha v-1} M_p^v(r,f)\,dr\\&\leq\beta v\,C^v\,\|f\|^{v-q}_{p,q,\alpha}\int_{0}^1(1-r)^{\alpha v-1}(1-r)^{-\alpha(v-q)} M_p^q(r,f)\,dr\\&=\frac{\beta v}{\alpha q}\,C^v\,\|f\|^{v}_{p,q,\alpha}.
\end{align*}

If $v=\infty,$ in a similar way, 
\begin{align*}\|f\|_{u,\infty,\beta}&=\sup_{0\leq r<1}(1-r)^\beta M_u(r,f)\leq C\sup_{0\leq r<1}(1-r)^\beta (1-r)^{\frac{1}{u}-\frac{1}{p}}M_p(r,f)\\&=C\sup_{0\leq r<1}(1-r)^\alpha M_p(r,f)=C\,\|f\|_{p,\infty,\alpha}\leq C\,\|f\|_{p,q,\alpha}.\end{align*}

Finally, we need to see that $H(p,q,\alpha)\not\subseteq H(u,v,\beta)$ when the parameters are not as in the assumptions of the statement. 
%if $\alpha+\frac{1}{p}>\beta+\frac{1}{u}$ or $\alpha+\frac{1}{p}=\beta+\frac{1}{u}$ and $q>v.$ 
If $\alpha+\frac{1}{p}>\beta+\frac{1}{u},$ Lemma~\ref{fg} tells us that $$f(z)=\frac{1}{(1-z)^{\beta+1/u}}$$ belongs to $H(p,q,\alpha)$ but not $H(u,v,\beta),$ and this proves that $H(p,q,\alpha)\not\subseteq H(u,v,\beta)$ when $\alpha+\frac{1}{p}>\beta+\frac{1}{u}.$

If $\alpha+\frac{1}{p}=\beta+\frac{1}{u}$ and $q>v,$ by Lemma~\ref{fc} the function $$f(z)=\frac{1}{(1-z)^{\alpha+1/p}}\left(\log\frac{e}{1-z}\right)^{-1/v}$$ is an example of a function in $H(p,q,\alpha)$ which is not in $H(u,v,\beta),$ and hence $H(p,q,\alpha)\not\subseteq H(u,v,\beta)$ for $\alpha+\frac{1}{p}=\beta+\frac{1}{u}$ and $q>v.$

\end{proof}

%\section*{Acknowledgments} This work was partially supported by MINECO grant MTM2012-37436-C02-02, Spain and is part of the author's doctoral thesis at Universidad Aut\'onoma de Madrid under the supervision of Professor Dragan Vukoti\'c. She wants to thank him for his useful comments and encouragement.

%\bibliography{mybib}{}
%\bibliographystyle{amsalpha}

\bibliographystyle{amsplain}

\end{document}